\documentclass[twoside,11pt]{article}
\usepackage{graphicx,amscd,amsfonts,amsmath,amsthm,amssymb,latexsym,amsfonts,color}
\usepackage{enumerate,framed,graphicx,graphics,algorithm,algorithmic,varwidth, pifont}
\usepackage{epsfig,float,epstopdf,array,bm,cite,mathrsfs}
\usepackage{cases,array,multirow}
\usepackage[table]{xcolor}
\usepackage{tikz,lmodern,mathtools,tikz,pgfplots,relsize,graphicx,float}
\usepackage[bookmarksnumbered, colorlinks,plainpages]{hyperref}
\newtheorem{theorem}{\bf Theorem}[section]
\newtheorem{lemma}[theorem]{\bf Lemma}

\newtheorem{example}[theorem]{\bf Example}

\newtheorem{corollary}[theorem]{\bf Corollary}

\begin{document}
	\title{\bf Inexact versions of several block-splitting preconditioners for indefinite least squares problems}
	\author{\small\bf  Mohaddese Kaveh Shaldehi$^\dag$,  Davod Khojasteh Salkuyeh$^{\dag}$\thanks{\noindent Corresponding author.}\\ 
	\textit{{\small $^\dag$Faculty of Mathematical Sciences, University of Guilan, Rasht, Iran}} \\
	\textit{{\small E-mails: mohaddese\_kaveh@phd.guilan.ac.ir, khojasteh@guilam.ac.ir, salkuyeh@gmail.com}}\\[2mm]
	}
	
	\date{}
	\maketitle
\vspace{-0.5cm}

	\noindent\hrulefill\\
	{\bf Abstract.} This paper introduces inexact versions of several block-splitting preconditioners for solving the  three-by-three block linear systems arising from a special class of indefinite least squares problems. We first establish the convergence conditions for the corresponding stationary iterative methods. Then, it follows that under these conditions, all eigenvalues of the preconditioned matrices are contained within a circle centered at $(1,0)$ with radius $1$. This property implies that these preconditioners are effective in accelerating the convergence of the GMRES method. Furthermore, we analyze the eigenpairs of the preconditioned matrices in detail and derive a theoretical upper bound on the number of GMRES iterations for solving the preconditioned systems. Ultimately, numerical experiments reveal the efficacy  of the proposed preconditioners.
	\\
		
		\noindent{\it  \footnotesize Keywords}: {\small Indefinite least squares problem, Block-splitting preconditioner, Stationary iterative method, Preconditioning techniques, Spectral clustering.
		
		\noindent\hrulefill\\
		
		\pagestyle{myheadings}\markboth{M. Kaveh  Shaldehi, D.K. Salkuyeh}{Inexact versions of several block-splitting preconditioners for ILS problems}
		\thispagestyle{empty}

\section{Introduction}\label{sec.1}
The following standard notations are used throughout this paper. The symbols $I$, $A^T$, and $A^*$ represent the identity matrix of an appropriate size, the transpose, and the conjugate transpose of a matrix $A$, respectively. The notation $\|.\|_2$ represents the Euclidean norm for vectors and the corresponding spectral norm for matrices, while $\|.\|_1$ denotes the
$1$-norm for both vectors and matrices. The spectrum, spectral radius, and null space of $A$ are denoted by $\sigma(A)$, $\rho(A)$,  and $\mathcal{N}(A)$, respectively. 
The spectral condition number of a nonsingular matrix $A$ is denoted by $\kappa(A)$,
which is defined as $\kappa(A)=\|A\|_2 \|A^{-1}\|_2.$
For two column vectors $x$ and $y$, we write $(x;y)$ to denote the column vector $(x^T, y^T)^{T}.$ Finally, the notation $A\succ 0$ indicates that the matrix  $A$ is symmetric positive definite (SPD). 

This paper is concerned with the indefinite least squares (ILS) problems of the form
		\begin{align}
		\min_{x \in \mathbb{R}^n} (b - A x)^T  H_{p,q} (b - A x),
		\label{eq:ILS}
	\end{align}
	where $A \in \mathbb{R}^{m \times n}$ is a given large and sparse matrix with $m \geq n$, $b \in \mathbb{R}^m$ is a given vector, and
	 $ H_{p,q}\in \mathbb{R}^{m \times m}$ is the signature matrix of the form
	\[  H_{p,q} = \begin{pmatrix}
		I_{p}&{\bf{0}}\\
		{\bf{0}}&-I_{q}
	\end{pmatrix},\qquad p+q=m,\]
where $I_p $ and $I_q $ denote the identity matrices of sizes $p$ and $q$, respectively.
Clearly, when either $p = 0$ or $q = 0$ the problem simplifies to the standard least squares (LS) problem. However, when $p,q>0$, the quadratic form in problem \eqref{eq:ILS} is indefinite.
Such problems occur in numerous applications, including the solution of total least squares (TLS) problems \cite{golub1980,vanhuffel1991,chandrasekaran1998}, and the area of optimization known as robust or $H^\infty$ smoothing \cite{hassibi1993,sayed1996,chandrasekaran1998}. 
	
The normal equations corresponding to the ILS problem \eqref{eq:ILS} can be written as
	 	\begin{equation}
		A^T  H_{p,q} A x = A^T  H_{p,q} b.
		\label{eq:normal}
	\end{equation}
	It is clear that unlike the classical LS problem, the ILS problem \eqref{eq:ILS} is not guaranteed to have a unique solution. In addition, the Hessian of its objective function is $2A^T  H_{p,q} A$. Consequently, there exists a   unique solution to the ILS problem \eqref{eq:ILS}  if and only if $A^T  H_{p,q} A$ is SPD. This condition is assumed to hold throughout this paper.
	
	Many research efforts have been focused on numerical solutions to the ILS problem. For small and dense instances of the problem \eqref{eq:ILS}, Chandrasekaran et al. \cite{chandrasekaran1998} introduced a direct method employing QR decomposition and Cholesky factorization. Subsequently, Bojanczyk et al. \cite{bojanczyk2003} developed an implementation of hyperbolic QR factorization, utilizing both Householder and hyperbolic Givens transformations. Furthermore, Xu \cite{xu2004} presented a backward stable hyperbolic QR factorization method. Nonetheless, for large-scale sparse problems, these methods become excessively costly because of their significant storage needs and computational expenses. Consequently, the development of efficient iterative methods, particularly Krylov subspace methods, has become a primary research focus in recent years.
		
	Among the iteration schemes, matrix-splitting methods represent a classical family of solvers for the ILS problem. For instance, Liu and Liu \cite{liu2014} proposed two block SOR-based algorithms for solving the system associated with the ILS problem. Subsequently, to accelerate convergence, Song \cite{song2020} introduced the Unsymmetric SOR (USSOR) method, which employs two independent relaxation parameters. However, the performance of the USSOR method is highly sensitive to the choice of its relaxation parameters and finding their optimal values remains a major challenge for large-scale systems.
	
	More recently, the range of iterative methods for the ILS problem has further diversified. Zhang and Li \cite{zhang2023} introduced randomized splitting methods. Subsequently, Meng et al. \cite{meng2024} constructed a variable-parameter Uzawa method.
	Furthermore, Meng et al. \cite{MengJIAM} developed an enhanced Landweber method with momentum acceleration. Additionally, Meng et al. \cite{MengAML} presented an alternating direction implicit (ADI) method.
 Most recently, Li and Xin \cite{Li2026} developed an alternating Anderson-variable parameter Uzawa method, which utilizes Anderson acceleration to enhance the convergence of the variable-parameter Uzawa method for the ILS problems.
 
	Another approach to solving the problem \eqref{eq:ILS} is to reformulate the normal equations \eqref{eq:normal} into a larger  three-by-three block linear system and solving it by a Krylov subspace methods such as GMRES \cite{SaadBook,GMRES}. These methods are widely favoured for their convergence properties; however, their performance on sparse linear systems can be significantly enhanced by suitable preconditioners. In this context, Xin and Meng \cite{Xin} recently proposed three distinct block-splitting preconditioners. Subsequently, Li et al.  introduced another preconditioner of this type in \cite{Li2025}. Moreover, Li and Meng \cite{Li2025gmres} employed the GMRES method with an accelerated preconditioner to solve the associated  three-by-three block linear system. Most recently, Xin et al. \cite{XinJAMC} applied a shift-splitting iteration method to solve this system and utilized the induced preconditioner. 
	 Furthermore, Salkuyeh \cite{Khojasteh2026} presented a parameterized block-splitting preconditioner to accelerate the convergence of GMRES for solving this system. In the following, we explain some of these block-splitting preconditioning techniques in more detail.

According to the structure of the matrix $H_{p,q}$, we partition the matrix $A$ and the vector $b$ as
	\[
	A = \begin{pmatrix} A_1 \\ A_2 \end{pmatrix}, \quad b = \begin{pmatrix} b_1 \\ b_2 \end{pmatrix},
	\]
	where the matrix  $A_1 \in \mathbb{R}^{p \times n}$  has full column rank, $A_2 \in \mathbb{R}^{q \times n}$, $b_1 \in \mathbb{R}^{p}$, and $b_2 \in \mathbb{R}^{q}$.
	The reformulation of the ILS problem \eqref{eq:ILS} as the following nonsingular  three-by-three block linear system was first proposed by Xin and Meng \cite{Xin}
		\begin{equation}
		\mathcal{A} \mathbf{x}=
		\begin{pmatrix}
			I & A_1 & {\bf{0}} \\
			{\bf{0}} & P & A_2^T \\
			{\bf{0}} & A_2 & I
		\end{pmatrix}
	\begin{pmatrix} \delta_1 \\ x \\ \delta_2 \end{pmatrix}=
	\begin{pmatrix} b_1 \\ A_1^Tb_1 \\ b_2 \end{pmatrix}
	 = \tilde{b},
		\label{eq:block3x3}
	\end{equation}
	where $P=A_1^TA_1$ and $\delta=(\delta_1;\delta_2)=b-Ax$. For this  three-by-three block linear system, they constructed three distinct block splittings:
	\begin{align}
		\mathcal{A} &= 
			\begin{pmatrix}
				I & {\bf{0}} & {\bf{0}} \\
				{\bf{0}} & P & {\bf{0}} \\
				{\bf{0}} & {\bf{0}} & I
			\end{pmatrix}
-
			\begin{pmatrix}
				{\bf{0}} & -A_1 & {\bf{0}} \\
				{\bf{0}} & {\bf{0}}& -A_2^T  \\
				{\bf{0}} & -A_2 & {\bf{0}}
			\end{pmatrix}
		:=\mathcal{M}_{BS1}-\mathcal{N}_{BS1},  \label{BS1} \\
		\mathcal{A} &= 
			\begin{pmatrix}
				I & {\bf{0}} & {\bf{0}} \\
				{\bf{0}} & P & A_2^T \\
				{\bf{0}} & {\bf{0}} & I
			\end{pmatrix}
		-
			\begin{pmatrix}
				{\bf{0}} & -A_1 & {\bf{0}} \\
				{\bf{0}} & {\bf{0}} & {\bf{0}} \\
				{\bf{0}} & -A_2 & {\bf{0}}
			\end{pmatrix}
		:=\mathcal{M}_{BS2}-\mathcal{N}_{BS2},  \label{BS2} \\
		\mathcal{A} &= 
			\begin{pmatrix}
				I & A_1 & {\bf{0}} \\
				{\bf{0}} & P & {\bf{0}} \\
				{\bf{0}} & {\bf{0}} & I
			\end{pmatrix}
		-
			\begin{pmatrix}
				{\bf{0}} & {\bf{0}} & {\bf{0}} \\
				{\bf{0}} &{\bf{0}} &  -A_2^T \\
				{\bf{0}} & -A_2 & {\bf{0}}
			\end{pmatrix}
		:=\mathcal{M}_{BS3}-\mathcal{N}_{BS3}.  \label{BS3} 
 	\end{align}
  	Subsequently, Li et al. \cite{Li2025} proposed a block upper triangular (BUT) splitting with the aim of minimizing the norm of the difference between the matrix $\mathcal{A}$ and the block-splitting preconditioner:
 		\begin{align}
		\mathcal{A} &= 
			\begin{pmatrix}
				I & A_1 & {\bf{0}} \\
				{\bf{0}} & P & A_2^T \\
				{\bf{0}} & {\bf{0}} & I
			\end{pmatrix}
		-
			\begin{pmatrix}
				{\bf{0}} & {\bf{0}} & {\bf{0}} \\
				{\bf{0}} &{\bf{0}}& {\bf{0}} \\
				{\bf{0}} &-A_2 & {\bf{0}}
			\end{pmatrix}
 		:=\mathcal{M}_{BUT}-\mathcal{N}_{BUT}. \label{BUT}
\end{align}

	Implementing the preconditioners $\mathcal{M}_{BS1}, \mathcal{M}_{BS2}, \mathcal{M}_{BS3}$, and $ \mathcal{M}_{BUT}$ requires solving inner linear systems with the coefficient matrix $P = A_1^T A_1$. The main issue is that the matrix $P$ is not guaranteed to be well-conditioned, especially when the matrix $A_1$ is ill-conditioned. Hence, in many large and practical ILS problems, the outer GMRES iterations may fail to reach a sufficiently accurate solution within a prescribed number of iterations, due to the accumulation of errors from solving the inner systems.	
	
	To compensate for this weakness, we introduce inexact versions of these block-splitting preconditioners. The main idea is to replace the matrix $P$ within the block-splitting preconditioners $\mathcal{M}_{BS1}, \mathcal{M}_{BS2}, \mathcal{M}_{BS3}$, and $ \mathcal{M}_{BUT}$, with a well-conditioned SPD approximation. Our theoretical analysis shows that, under the convergence conditions of the corresponding stationary iterative methods, all the eigenvalues of the preconditioned matrices are located within a circle centered at $(1,0)$ with radius 1. This robust spectral clustering leads to accelerated convergence of the GMRES method. Furthermore, numerical experiments demonstrate the efficiency and robustness of the proposed preconditioners in comparison with recent methods.
	
	The paper is organized as follows. Section \ref{sec.2} introduces our new block-splitting preconditioners, the convergence analysis of the corresponding stationary iterative methods, and implementation of the proposed preconditioners. In Section \ref{sec.3}, we conduct a detailed spectral analysis of the preconditioned matrices; specifically, we analyze their eigenpairs and derive a theoretical upper bound on the number of GMRES iterations. Section \ref{sec.4} presents some numerical experiments to illustrate the efficiency of the proposed preconditioners. Finally, Section \ref{sec.5} offers some concluding remarks.
	
	\section{The inexact block-splitting preconditioners} \label{sec.2}		
	In this section, we introduce the inexact block-splitting (IBS) preconditioners. The construction of these preconditioners is based on replacing the matrix $P$ in the block-splitting preconditioners $\mathcal{M}_{BS1}$, $\mathcal{M}_{BS2}$, $\mathcal{M}_{BS3}$, and $\mathcal{M}_{BUT}$, with a well-conditioned SPD approximation, $\hat{P}.$ In other words, the IBS preconditioners are derived from the following matrix splittings
		\begin{align}
			\mathcal{A} &= 
			\begin{pmatrix}
				I & {\bf{0}} & {\bf{0}} \\
				{\bf{0}} & \hat{P} & {\bf{0}} \\
				{\bf{0}} & {\bf{0}} & I
			\end{pmatrix}
			-
			\begin{pmatrix}
				{\bf{0}} & -A_1 & {\bf{0}} \\
				{\bf{0}} & \hat{P}-P& -A_2^T  \\
				{\bf{0}} & -A_2 & {\bf{0}}
			\end{pmatrix}
			:=\mathcal{M}_1-\mathcal{N}_1,  \label{IBS1} \\
			\mathcal{A} &= 
			\begin{pmatrix}
				I & {\bf{0}} & {\bf{0}} \\
				{\bf{0}} &\hat{P} & A_2^T \\
				{\bf{0}} & {\bf{0}} & I
			\end{pmatrix}
			-
			\begin{pmatrix}
				{\bf{0}} & -A_1 & {\bf{0}} \\
				{\bf{0}} & \hat{P}-P & {\bf{0}} \\
				{\bf{0}} & -A_2 & {\bf{0}}
			\end{pmatrix}
			:=\mathcal{M}_2-\mathcal{N}_2, \label{IBS2}\\
			\mathcal{A} &= 
			\begin{pmatrix}
				I & A_1 & {\bf{0}} \\
				{\bf{0}} & \hat{P} & {\bf{0}} \\
				{\bf{0}} & {\bf{0}} & I
			\end{pmatrix}
			-
			\begin{pmatrix}
				{\bf{0}} & {\bf{0}} & {\bf{0}} \\
				{\bf{0}} &\hat{P}-P &  -A_2^T \\
				{\bf{0}} & -A_2 & {\bf{0}}
			\end{pmatrix}
			:=\mathcal{M}_3-\mathcal{N}_3,\label{IBS3}\\
			\mathcal{A} &= 
			\begin{pmatrix}
				I & A_1 & {\bf{0}} \\
				{\bf{0}} & \hat{P} & A_2^T \\
				{\bf{0}} & {\bf{0}} & I
			\end{pmatrix}
			-
			\begin{pmatrix}
				{\bf{0}} & {\bf{0}} & {\bf{0}} \\
				{\bf{0}} &\hat{P}-P& {\bf{0}} \\
				{\bf{0}} &-A_2 & {\bf{0}}
			\end{pmatrix}
			:=\mathcal{M}_4-\mathcal{N}_4. \label{IBS4}
		\end{align}
		The IBS preconditioners correspond to the matrices $\mathcal{M}_1,  \mathcal{M}_2,  \mathcal{M}_3$, and $ \mathcal{M}_4$, which are termed the IBS1, IBS2, IBS3, and IBS4 preconditioners, respectively.
		\subsection{The IBS iterative methods}
		Each of the proposed block splittings induces a stationary iterative method for solving the  three-by-three block linear system \eqref{eq:block3x3}. We term them the IBS iterative methods, which take the corresponding general form
		\begin{align}
			\mathbf{x}^{(k+1)}=\mathcal{G}_i\mathbf{x}^{(k)} + \mathcal{M}_i^{-1}\tilde{b}, \qquad i=1, 2, 3, 4,
			\label{eq:stationary} 
		\end{align}
		 where $\mathbf{x}^{(0)} \in \mathbb{R}^{m+n}$ is an arbitrary initial guess and  $\mathcal{G}_i=\mathcal{M}_i^{-1} \mathcal{N}_i=I-\mathcal{M}_i^{-1}\mathcal{A}$, $i=1,2,3,4$, are the corresponding iteration matrices. 
		 
		 Recalling that the matrices $P, \hat{P}$, and $A^T H_{p,q}A=A_1^T A_1 - A_2^T A_2$ are all SPD, we begin to analyze the convergence of the IBS iterative methods. For this analysis, we will need the following well-known lemma. 

	\begin{lemma}\cite{AxelBook}\label{Lem1}
		Let $\alpha$ and $\beta$ be the roots of the real quadratic equation $x^2-bx+c=0$. The modulus of both $\alpha$ and $\beta$ is less than one   if and only if $$
		|c|<1,\quad |b|<1+c.$$  
	\end{lemma}	
	
	 \begin{theorem}\label{Theo1}
	 Let $\mathbf{x}^*$ be the unique solution to the linear system \eqref{eq:block3x3}. The IBS stationary iterative methods have the following convergence properties:
	 \begin{enumerate}[{\rm (i)}]
	\item The IBS1 and IBS3 methods converge to $\mathbf{x}^*$ for any initial guess if both the matrices $2\hat{P} - P - A_2^T A_2$ and $\hat{P} - A_2^T A_2$ are SPD, i.e.,
	$$
	2\hat{P} - P - A_2^T A_2 \succ 0 \quad \text{and} \quad \hat{P} - A_2^T A_2 \succ 0 \implies \rho(G_1) < 1 \text{ and } \rho(G_3) < 1.
	$$
	\item The IBS2 and IBS4 methods converge to $\mathbf{x}^*$ for any initial guess if the matrix $2\hat{P} - P + A_2^T A_2$ is SPD, i.e.,
	$$
	2\hat{P} - P + A_2^T A_2 \succ 0 \implies \rho(G_2) < 1 \text{ and } \rho(G_4) < 1.
	$$
\end{enumerate}
\end{theorem}
\begin{proof}
For the zero eigenvalues of the iteration matrices, the proof of the theorem is obvious. Therefore, we prove each part of the theorem separately, only for nonzero eigenvalues of these matrices.
  
\vspace{0.5em}  
\noindent{\small Proof of (i).} 
Let $(\lambda, \nu)$ be an eigenpair of the iteration matrix $\mathcal{G}_1 = \mathcal{M}_1^{-1} \mathcal{N}_1$, where $\lambda\neq0$ and $\nu = (x; y; z) \neq {\bf{0}}$. The eigenvalue equation $\mathcal{N}_1 \nu = \lambda \mathcal{M}_1 \nu$ yields
\begin{align}\label{conv1.1}
	\begin{cases}
		-A_1 y=\lambda x,\\
		(\hat{P}-P)y - A_2^T z=\lambda \hat{P}y,\\
		-A_2y=\lambda z.
	\end{cases}
\end{align} 
Note that for $\lambda\neq0$, we have $y\neq{\bf{0}}$, otherwise, the first and third equation of \eqref{conv1.1} implies $x={\bf{0}}$ and $z={\bf{0}}$, contradicting that $\nu \neq{\bf{0}}$ is an eigenvector. Therefore, substituting the third equation of \eqref{conv1.1} into the second one, and pre-multiplying the resulting equation by $\lambda y^*$, yields
	$$  (y^* \hat{P} y) \lambda^2- (y^* (\hat{P}-P) y)\lambda  - y^* A_2^T A_2 y = 0. $$
Since $\hat{P}$ is SPD, both sides of the above equation can be divided by $y^* \hat{P} y$. Therefore, we have
	\begin{align} \label{eq:quadratic_final}
		\lambda^2 - \frac{y^*(\hat{P}-P)y}{y^* \hat{P}y} \lambda -  \frac{y^*A_2^T A_2 y}{y^* \hat{P}y}  = 0.
	\end{align}
From Lemma \ref{Lem1}, the roots of this quadratic equation satisfy $|\lambda| < 1$ if and only if
	\begin{align}
		\frac{y^*A_2^T A_2 y}{y^* \hat{P}y} < 1, \label{cond1}
	\end{align}
and
	\begin{align}
		\left| \frac{y^*(\hat{P}-P)y}{y^* \hat{P}y} \right| < 1 - \frac{y^*A_2^T A_2 y}{y^* \hat{P}y}. \label{cond2}
	\end{align}
Inequality \eqref{cond1} is equivalent to $y^*(\hat{P} - A_2^T A_2)y > 0$, which is satisfied if the matrix $\hat{P} - A_2^T A_2$ is SPD. 
Inequality \eqref{cond2} can be rewritten as
	$$ -y^*(\hat{P} - A_2^T A_2)y < y^*(\hat{P}-P)y < y^*(\hat{P} - A_2^T A_2)y. $$
The right-hand inequality simplifies to $$y^*(P - A_2^T A_2)y > 0.$$ Since we assumed $P - A_2^T A_2$ is SPD, this inequality always holds true. The left-hand inequality simplifies to $$y^*(2\hat{P} - P - A_2^T A_2)y > 0,$$ which is satisfied if $2\hat{P} - P - A_2^T A_2$ is SPD.
This completes the proof for the IBS1 method. For the IBS3 method, the result can be obtained similarly, as in the IBS1 method.
	\vspace{1em}
	
\noindent{\small Proof of (ii).} 
Let $(\lambda, \nu)$ be an eigenpair of $\mathcal{G}_2 = \mathcal{M}_2^{-1} \mathcal{N}_2$, where $\lambda\neq0$ and $\nu = (x; y; z) \neq {\bf{0}}$. The eigenvalue equation $\mathcal{N}_2 \nu = \lambda \mathcal{M}_2 \nu$ can be rewritten as
		\begin{align}\label{conv1.6}
		\begin{cases}
			-A_1 y=\lambda x,\\
			(\hat{P}-P)y =\lambda (\hat{P}y+A_2^T z),\\
			-A_2y=\lambda z.
		\end{cases}
	\end{align} 
Substituting the third equation of \eqref{conv1.6} into the second one, yields
$$(\hat{P}-P+A_2^TA_2)y =\lambda\hat{P}y,$$
By pre-multiplying the above equation by $y^*$, we can deduce that
	\begin{align}\label{lambda}
	 \lambda = \frac{y^*(\hat{P}-P+A_2^TA_2)y}{y^*\hat{P}y}. 
	\end{align}
Thus, the condition $|\lambda| < 1$ is equivalent to
	$$ -y^*\hat{P}y < y^*(\hat{P}-P+A_2^TA_2)y < y^*\hat{P}y. $$
The right-hand inequality is equivalent to $$y^*(P-A_2^TA_2)y > 0.$$ The above inequality always holds since $P-A_2^TA_2 $ is assumed to be SPD. The left-hand inequality is equivalent to $$y^*(2\hat{P}-P+A_2^TA_2)y > 0,$$ which is satisfied if $2\hat{P}-P+A_2^TA_2$ is SPD.
The proof for the IBS4 method is similar to that of the IBS2 method. Thus, the proof is complete.
\end{proof}
\noindent From Theorem \ref{Theo1}, we can obtain the following corollary for the IBS iterative methods.
\begin{corollary}\label{coroll1}
	If $\hat{P}-P$ is a symmetric positive semidefinite (SPSD) matrix, then all the IBS stationary iterative methods converge to the unique solution of \eqref{eq:block3x3} for any initial guess. Furthermore, all the eigenvalues of $\mathcal{G}_2$ and $\mathcal{G}_4$ are contained within the positive real interval $(0,1)$.
\end{corollary}
\begin{proof}
Since the matrices $P-A_2^T A_2$  and $\hat{P}-P$ are, respectively,   SPD and  SPSD, we deduce that
\begin{align*}
	\hat{P}-A_2^T A_2=(\hat{P}-P)+(P-A_2^TA_2)\succ 0,
\end{align*}
and
\begin{align*}
	2\hat{P}-P-A_2^TA_2=2(\hat{P}-P)+(P-A_2^TA_2)\succ0.
\end{align*}
Consequently, since $A_2^T A_2$ is also SPSD, we have
$$2\hat{P}-P+A_2^{T}A_2=(2\hat{P}-P-A_2^TA_2)+2A_2^TA_2\succ0.$$
Therefore, all conditions of Theorem \ref{Theo1} are obtained, i.e., $\rho(\mathcal{G}_i)<1$, $i=1,2,3,4$.
In addition, the matrices $\hat{P}-P+A_2^TA_2$ and $\hat{P}$ are both SPD. Thus, from \eqref{lambda} we conclude that all the eigenvalues of $\mathcal{G}_2$ and $\mathcal{G}_4$ are located in the positive real interval $(0,1)$. This gives the desired result.
\end{proof}
 \subsection{Implementation details}
We have already analyzed the convergence of the IBS iterative methods and proved that if the conditions of Theorem \ref{Theo1} are satisfied, then for each iteration matrix $\mathcal{G}\in \{\mathcal{G}_1, \mathcal{G}_2, \mathcal{G}_3, \mathcal{G}_4 \}$, we have
$\rho(\mathcal{G})<1.$
This implies that for the corresponding preconditioner $\mathcal{M}\in \{\mathcal{M}_1, \mathcal{M}_2, \mathcal{M}_3, \mathcal{M}_4 \}$,  the eigenvalues of the preconditioned matrix
$\mathcal{M}^{-1}\mathcal{A}=I-\mathcal{G}$
are clustered within a circle centered at $(1,0)$ with radius $1$. This spectral clustering ensures that, the IBS preconditioners are effective in accelerating the convergence of a Krylov subspace method, such as GMRES, for solving the  three-by-three block linear system \eqref{eq:block3x3}. Applying such methods to solve the preconditioned system
$
	\mathcal{M}^{-1}\mathcal{A}\mathbf{x}	=\mathcal{M}^{-1}\tilde{b}
$
requires, at each iteration, solving a system of the form
$
	\mathcal{M}z=r.
$
To detail the implementation of this procedure, we partition the vectors $z$ and $r$ as $z=(z_1;z_2;z_3)$ and $r=(r_1;r_2;r_3).$ The algorithms for solving such systems are then given as follows:

\bigskip

\textbf{Algorithm 1.} Solving  $\mathcal{M}_1 z=r$ \\[-0.5cm]
\begin{itemize}
\item[1.]  $z_1:=r_1$;
\item[2.] Solve $\hat{P}z_2=r_2$ for $z_2$;
\item[3.] $z_3:=r_3.$
\end{itemize} 

\bigskip

\textbf{Algorithm 2.} Solving  $\mathcal{M}_2 z=r$ \\[-0.5cm]
\begin{itemize}
\item [1.] $z_1:=r_1$;
\item[2.]  $z_3:=r_3$;
\item[3.] Solve $\hat{P}z_2=r_2-A_2^Tz_3$ for $z_2$.
\end{itemize} 

\bigskip

\textbf{Algorithm 3.} Solving  $\mathcal{M}_3 z=r$ \\[-0.5cm]
\begin{itemize}
	\item [1.] Solve $\hat{P}z_2=r_2$ for $z_2$;
	\item[2.]  $z_1:=r_1-A_1z_2$;
	\item[3.] $z_3:=r_3.$
\end{itemize} 

\bigskip

\textbf{Algorithm 4.} Solving  $\mathcal{M}_4 z=r$ \\[-0.5cm]
\begin{itemize}
	\item [1.] $z_3:=r_3$;
	\item[2.]  Solve $\hat{P}z_2=r_2-A_2^Tz_3$ for $z_2$;
	\item[3.] $z_1:=r_1-A_1z_2$.
\end{itemize} 

\bigskip

According to the above algorithms, it follows that the implementation of the IBS preconditioners, necessitates solving inner linear systems with the coefficient matrix $\hat{P}$, which is SPD. Thus, the inner systems can be solved exactly using the Cholesky factorization and inexactly using the conjugate gradient (CG) method. 

\subsection{Choosing the matrix $\hat{P}$}

The matrix $\hat{P}$ should be chosen such that it is more well-conditioned than the matrix  $P=A_1^TA_1$. For this purpose, we set $\hat{P}=\alpha I+P$ for some positive number $\alpha$. In this case, it is clear that $\hat{P}-P=\alpha I$ is an SPD matrix. Thus, it follows from Corollary \ref{coroll1} that all the IBS iterative methods would be convergent. 

It is worth mentioning that adding a positive value $\alpha$ to the main diagonal of $P$ improves its condition number. In fact, we have
\[
\kappa(\hat{P})=\frac{\alpha+\lambda_{\max}(P)}{\alpha+\lambda_{\min}(P)} < \frac{\lambda_{\max}(P)}{\lambda_{\min}(P)} =  \kappa(P).
\] 
To be more specific, if the matrix $P$ is scaled such that its largest eigenvalue equals 1, then for the spectral condition number of $\hat{P}=\alpha I+P$ we have (see \cite{BenziSD})
\[
\kappa(\hat{P})=\frac{\alpha+1}{\alpha+\lambda_{\min}(P)} <\frac{\alpha+1}{\alpha}= 1+\frac{1}{\alpha}.
\]
For example, if we set $\alpha=0.2$, then the spectral condition number of $\hat{P}$ is bounded by $6$. 
In practice, a practical choice for the parameter $\alpha$, as suggested in \cite{BenziActa}, is given by 
\[
\alpha=\|A_1\|^2,
\]
where $ \| . \|$ is a matrix norm. This choice is intended to balance the scale between the matrices $I$ and $A_1^TA_1$, and it will be used for all the numerical examples throughout this paper.

\section{Spectral analysis of the preconditioned matrices}\label{sec.3}

In this section, we investigate the spectral properties of the IBS preconditioned matrices. First, in the following four theorems, we analyze the structure of the eigenpairs of these matrices.
\begin{theorem} \label{Theorprec1}
	Assume that the conditions of Theorem \ref{Theo1} hold for the IBS1 method. Then, the preconditioned matrix $\mathcal{M}_1^{-1}\mathcal{A}$ has $p+q-\text{rank}(A_2)+h$ linearly independent eigenvectors and the structure of its eigenpairs has the following properties:
	\begin{enumerate}[{\rm (i)}]
		\item The eigenvalue $1$ has a geometric multiplicity of $p+q-\text{rank}(A_2)$. The corresponding linearly independent eigenvectors include $p$ vectors of the form $(x_i;{\bf{0}};{\bf{0}})$ and $q-\text{rank}(A_2)$ vectors of the form $({\bf{0}};{\bf{0}};z_i)$, where $\{x_i\}_{i=1}^{p}$ and $\{z_i\}_{i=1}^{q-\text{rank}(A_2)}$ are bases for $\mathbb{R}^p$ and $\mathcal{N}(A_2^T)$, respectively;
		
\item Let $h$ $(0\leq h \leq n)$ be the total number of linearly independent eigenvectors associated with the non-unit eigenvalues. Then, each non-unit eigenvalue, $\mu_j ,$ has a geometric multiplicity $h_j$, where $ 0\leq h_j \leq h \leq n $. The corresponding $h_j$ linearly independent eigenvectors are of the form $ \left(\frac{1}{\mu_j-1}A_1y_i; y_i; \frac{1}{\mu_j-1}A_2y_i\right),$ where $\{y_i\}_{i=1}^{h_j}$ is a set of linearly independent vectors satisfying the relation $ (1-\mu_j)(P-\mu_j\hat{P})y_i = A_2^T A_2 y_i. $ 
\end{enumerate}
\end{theorem}
\begin{proof}
	Let $(\mu, \nu)$ be an eigenpair of the preconditioned matrix $\mathcal{M}_1^{-1}\mathcal{A}$, where $\nu=(x;y;z)\neq {\bf{0}}$. The eigenvalue equation $\mathcal{M}_1^{-1}\mathcal{A}\nu = \mu\nu$ can be equivalently written in the  three-by-three block form as
	\begin{align} \label{precsystem1block}
		\begin{pmatrix}
			I & A_1 &{\bf{0}} \\
			{\bf{0}} & P & A_2^T \\
			{\bf{0}} & A_2 & I
		\end{pmatrix}
		\begin{pmatrix}
			x \\ y \\ z
		\end{pmatrix}
		= \mu
		\begin{pmatrix}
			I & {\bf{0}} & {\bf{0}} \\
			{\bf{0}} & \hat{P} & {\bf{0}} \\
			{\bf{0}} & {\bf{0}} & I
		\end{pmatrix}
		\begin{pmatrix}
			x \\ y \\ z
		\end{pmatrix},
	\end{align}
The above equation can be equivalently rewritten as
	\begin{align} \label{precsystem1}
		\begin{cases}
			A_1 y=(\mu -1)x, \\
			(P - \mu\hat{P})y= -A_2^T z,\\
			A_2 y=(\mu-1)z.
		\end{cases}
	\end{align}
	We now analyze the system \eqref{precsystem1} for the cases $\mu=1$ and $\mu \neq 1$.
	
	If $\mu=1$, then we have
	\begin{align} \label{precsystem1.1}
		\begin{cases}
			A_1 y = {\bf{0}}, \\
			(P-\hat{P})y = -A_2^T z, \\
			A_2 y={\bf{0}}.
		\end{cases}
	\end{align}
Since $A_1$ has a full column rank, the first equation of \eqref{precsystem1.1} implies that $y={\bf{0}}$. Substituting this into the second equation yields $A_2^T z = {\bf{0}}$, which means $z \in \mathcal{N}(A_2^T)$. Therefore, the linearly independent eigenvectors corresponding to the eigenvalue $1$ include $p$ vectors of the form $(x_i;{\bf{0}};{\bf{0}})$, where $\{x_i\}_{i=1}^{p}$ is a basis for $\mathbb{R}^p$, and $q-\text{rank}(A_2)$ vectors of the form $({\bf{0}};{\bf{0}};z_i)$, where $\{z_i\}_{i=1}^{q-\text{rank}(A_2)}$ is a basis for $\mathcal{N}(A_2^T)$. Consequently, the geometric multiplicity of the eigenvalue $1$ is $p+q-\text{rank}(A_2)$.

For $\mu \neq 1$, from \eqref{precsystem1} we have
\begin{align}\label{precsystem1.2}
\begin{cases}
x = \frac{1}{\mu-1} A_1 y, \\
(P - \mu\hat{P})y =-A_2^Tz,\\
z = \frac{1}{\mu-1} A_2 y.
\end{cases}
\end{align}
	Note that $y \neq {\bf{0}}$, otherwise we have $x={\bf{0}}$ and $z={\bf{0}}$, which contradicts that $\nu$ is an eigenvector. By substituting $z=\frac{1}{\mu-1} A_2 y$ into the second equation of \eqref{precsystem1.2}, we obtain
\begin{align}\label{relationprec1}
	(1-\mu)(P - \mu\hat{P})y = A_2^T A_2 y.
\end{align}
	Therefore, there exist $h$ ($0 \leq h \leq n$) linearly independent eigenvectors of the form
	$$ \left( \frac{1}{\mu-1}A_1y_i; y_i; \frac{1}{\mu-1}A_2y_i \right), $$
	corresponding to non-unit eigenvalues $\mu \neq 1$, where the vectors $\{y_i\}_{i=1}^h$ are linearly independent vectors satisfying the relation \eqref{relationprec1}.	In fact, each distinct non-unit eigenvalue, say $\mu_j$, corresponds to $h_j$ of these $h$ linearly independent eigenvectors. Therefore, $h_j$ is its geometric multiplicity.

The proof of the independence of eigenvectors is straightforward and omitted here for conciseness.
\end{proof}

\begin{theorem} \label{Theorprec2}
	Assume that the conditions of Theorem \ref{Theo1} hold for the IBS2 method. Then, the preconditioned matrix $\mathcal{M}_2^{-1}\mathcal{A}$ has $p+q+h$ linearly independent eigenvectors and the structure of its eigenpairs has the following properties:
	\begin{enumerate}[{\rm (i)}]
		\item The eigenvalue $1$ has a geometric multiplicity of  $p+q$. The corresponding linearly independent eigenvectors are of the form $(x_i;{\bf{0}};z_i)$, where $\{x_i\}_{i=1}^{p}$ and $\{z_i\}_{i=1}^{q}$ are bases of $\mathbb{R}^p$ and $\mathbb{R}^q$, respectively; 
		
		\item The total number of its linearly independent eigenvectors associated with the non-unit eigenvalues is $h$ $(0\leq h \leq n)$, where each eigenvalue, $\mu_j \neq 1,$ has a geometric multiplicity $h_j$ $ (0\leq h_j \leq h \leq n) $. The corresponding $h_j$ linearly independent eigenvectors are of the form $\left( \frac{1}{\mu_j-1}A_1y_i; y_i; \frac{1}{\mu_j-1}A_2y_i\right) $, where $\{y_i\}_{i=1}^{h_j}$ is a set of linearly independent vectors satisfying the eigenvalue problem $\hat{P}^{-1}(P-A_2^TA_2)y_i = \mu_j y_i; $ 
		
		\item All eigenvalues of the preconditioned matrix $\mathcal{M}_2^{-1}\mathcal{A}$ are located in the positive real interval $(0,2)$.
	\end{enumerate}
\end{theorem}
\begin{proof}
	Since the proofs of parts (i) and (ii) of this theorem are  similar to those of Theorem \ref{Theorprec1}, they are not repeated here. 
	Thus, we only provide the proof for part (iii). 
	We know that each eigenvalue $\mu$ of the preconditioned matrix $\mathcal{M}_2^{-1}\mathcal{A}$ is related to a corresponding eigenvalue $\lambda$ of the iteration matrix $\mathcal{G}_2$ by the equation $\mu=1-\lambda$. Thus, under the conditions of Theorem \ref{Theo1}, $\mu$ lies in the disk centered at $(1,0)$ with radius $1$.
	
	The proof of (iii) is obvious for the case $\mu=1$. Furthermore, from part (ii), it can be deduced that $\mu\neq 1$ is an eigenvalue of $\hat{P}^{-1}(P-A_2^TA_2)$ corresponding to an eigenvector $y$. Therefore, $\hat{P}^{-1}(P-A_2^TA_2)y=\mu y.$ This is equivalent to 
	$$ (P-A_2^TA_2)y= \mu \hat{P}y. $$
	Pre-multiplying both sides of the above equation by $y^*$ gives
	$$ \mu = \frac{y^*(P-A_2^TA_2) y}{y^*\hat{P}y}. $$
	Since $P-A_2^TA_2$ and $\hat{P}$ are both SPD, $\mu$ is positive and located in the real interval $(0,2).$ Thus, the proof is completed. 
\end{proof}

\begin{theorem} \label{Theorprec3}
	Assume that the conditions of Theorem \ref{Theo1} hold for the IBS3 method. Then, the preconditioned matrix $\mathcal{M}_3^{-1}\mathcal{A}$ has $p+q+n-2\text{rank}(A_2)+h$ linearly independent eigenvectors and the structure of its eigenpairs has the following properties:
\begin{enumerate}[{\rm (i)}]
			\item The eigenvalue $1$ has a geometric multiplicity of  $p+q+n-2\text{rank}(A_2)$. The corresponding linearly independent eigenvectors include $p$ vectors of the form $(x_i;{\bf{0}};{\bf{0}})$, $n-\text{rank}(A_2)$ vectors of the form $({\bf{0}};y_i;{\bf{0}})$, and $q-\text{rank}(A_2)$ vectors of the form $({\bf{0}};{\bf{0}};z_i)$, where $\{x_i\}_{i=1}^{p}$, $\{y_i\}_{i=1}^{n-\text{rank}(A_2)}$, and $\{z_i\}_{i=1}^{q-\text{rank}(A_2)}$ are bases for $\mathbb{R}^p$, $\mathcal{N}(A_2)$ and $\mathcal{N}(A_2^T)$, respectively; 
			
		\item The total number of its linearly independent eigenvectors associated with the non-unit eigenvalues is $h$ $(0\leq h \leq n)$, where each eigenvalue, $\mu_j \neq 1,$ has a geometric multiplicity $h_j$ $ (0\leq h_j \leq h \leq n) $. The corresponding $h_j$ linearly independent eigenvectors are of the form $\left( -A_1y_i; y_i; \frac{1}{\mu_j-1}A_2y_i\right) $, where $\{y_i\}_{i=1}^{h_j}$ is a set of linearly independent vectors satisfying the relation $ (1-\mu_j)(P-\mu_j\hat{P})y_i = A_2^T A_2 y_i. $ 
\end{enumerate}
\end{theorem}

\begin{theorem} \label{Theorprec4}
	Assume that the conditions of Theorem \ref{Theo1} hold for the IBS4 method. Then, the preconditioned matrix $\mathcal{M}_4^{-1}\mathcal{A}$ has $p+q+\dim(\mathcal{N}(\hat{P}-P)\cap \mathcal{N}(A_2))+h$ linearly independent eigenvectors and the structure of its eigenpairs has the following properties:
	\begin{enumerate}[{\rm (i)}]
		\item The eigenvalue $1$ has a geometric multiplicity of  $p+q+\dim(\mathcal{N}(\hat{P}-P)\cap \mathcal{N}(A_2))$. The corresponding linearly independent eigenvectors include $p+q$ vectors of the form $(x_i;{\bf{0}};z_i)$ and $\dim(\mathcal{N}(\hat{P}-P)\cap \mathcal{N}(A_2))$ vectors of the form $({\bf{0}};y_i;{\bf{0}})$, where $\{x_i\}_{i=1}^{p}$, $\{y_i\}_{i=1}^{\dim(\mathcal{N}(\hat{P}-P)\cap \mathcal{N}(A_2))}$, and $\{z_i\}_{i=1}^{q}$ are bases of $\mathbb{R}^p$, $\mathcal{N}((\hat{P}-P)\cap \mathcal{N}(A_2))$, and $\mathbb{R}^q,$ respectively;
		
		\item The total number of its linearly independent eigenvectors associated with the non-unit eigenvalues is $h$ $(0\leq h \leq n)$, where each eigenvalue, $\mu_j \neq 1,$ has a geometric multiplicity $h_j$ $ (0\leq h_j \leq h \leq n) $. The corresponding $h_j$ linearly independent eigenvectors are of the form $\left( -A_1y_i; y_i; \frac{1}{\mu_j-1}A_2y_i\right) $, where $\{y_i\}_{i=1}^{h_j}$ is a set of linearly independent vectors satisfying the eigenvalue problem $\hat{P}^{-1}(P-A_2^TA_2)y_i = \mu_j y_i $;
		
		\item All eigenvalues of the preconditioned matrix $\mathcal{M}_4^{-1}\mathcal{A}$ are located in the positive real interval $(0,2).$ 
\end{enumerate}
\end{theorem}
Since the proofs of Theorems \ref{Theorprec3} and \ref{Theorprec4} are similar to those of  Theorems \ref{Theorprec1} and  \ref{Theorprec2}, respectively, they are not presented here to avoid repetition. In the following theorem, we derive an upper bound on the degree of the minimal polynomial for any of the IBS preconditioned matrices.
\begin{theorem}\label{Theorminimal}
	Let the IBS preconditioners $\mathcal{M}_i$, $i=1,2,3,4$, be defined as in \eqref{IBS1}, \eqref{IBS2}, \eqref{IBS3}, and \eqref{IBS4}. Then, for each of the preconditioned matrices $\mathcal{M}_i^{-1}\mathcal{A}$, $i=1,2,3,4$, the minimal polynomial has a degree of at most $n+q+1$.
\end{theorem}
\begin{proof}
The  three-by-three block structures of the IBS preconditioned matrices are obtained as follows: 
	\begin{align}
		\mathcal{M}_1^{-1} \mathcal{A}&=
		\begin{pmatrix}
			I & {\bf{0}} & {\bf{0}} \\
			{\bf{0}} & \hat{P}^{-1} &{\bf{0}}  \\
			{\bf{0}} & {\bf{0}} & I
		\end{pmatrix}
		\begin{pmatrix}
			I & A_1 & {\bf{0}} \\
			{\bf{0}} & P & A_2^T \\
			{\bf{0}} & A_2 & I
		\end{pmatrix}=
		\begin{pmatrix}
			I&A_1&{\bf{0}}\\
			{\bf{0}}&\hat{P}^{-1}P&\hat{P}^{-1}A_2^T\\
			{\bf{0}}&A_2&I
		\end{pmatrix},\label{precmat1}\\
		\mathcal{M}_2^{-1} \mathcal{A}&=
		\begin{pmatrix}
			I & {\bf{0}} & {\bf{0}} \\
			{\bf{0}} & \hat{P}^{-1} &-\hat{P}^{-1}A_2^T  \\
			{\bf{0}} & {\bf{0}} & I
		\end{pmatrix}
		\begin{pmatrix}
			I & A_1 & {\bf{0}} \\
			{\bf{0}} & P & A_2^T \\
			{\bf{0}} & A_2 & I
		\end{pmatrix}=
		\begin{pmatrix}
			I&A_1&{\bf{0}}\\
			{\bf{0}}&\hat{P}^{-1}(P-A_2^TA_2)&{\bf{0}}\\
			{\bf{0}}&A_2&I
		\end{pmatrix},\label{precmat2}\\
		\mathcal{M}_3^{-1} \mathcal{A}&=
		\begin{pmatrix}
			I & -A_1\hat{P}^{-1} & {\bf{0}} \\
			{\bf{0}} & \hat{P}^{-1} &{\bf{0}}  \\
			{\bf{0}} & {\bf{0}} & I
		\end{pmatrix}
		\begin{pmatrix}
			I & A_1 & {\bf{0}} \\
			{\bf{0}} & P & A_2^T \\
			{\bf{0}} & A_2 & I
		\end{pmatrix}=
		\begin{pmatrix}
			I&A_1(I-\hat{P}^{-1}P)&-A_1\hat{P}^{-1}A_2^T\\
			{\bf{0}}&\hat{P}^{-1}P&\hat{P}^{-1}A_2^T\\
			{\bf{0}}&A_2&I
		\end{pmatrix},\label{precmat3}\\
		\mathcal{M}_4^{-1} \mathcal{A}&=
		\begin{pmatrix}
			I & -A_1\hat{P}^{-1} &A_1\hat{P}^{-1} A_2^T \\
			{\bf{0}} & \hat{P}^{-1} &-\hat{P}^{-1}A_2^T  \\
			{\bf{0}} & {\bf{0}} & I
		\end{pmatrix}
		\begin{pmatrix}
			I & A_1 & {\bf{0}} \\
			{\bf{0}} & P & A_2^T \\
			{\bf{0}} & A_2 & I
		\end{pmatrix}=
		\begin{pmatrix}
			I&A_1(I-\hat{P}^{-1}(P-A_2^TA_2))&{\bf{0}}\\
			{\bf{0}}&\hat{P}^{-1}(P-A_2^TA_2)&{\bf{0}}\\
			{\bf{0}}&A_2&I
		\end{pmatrix}.\label{precmat4}
	\end{align}
We can see that each IBS preconditioned matrix, denoted by $\mathcal{M}^{-1}\mathcal{A}$, gives a common  three-by-three block structure of the form
\begin{align}\label{commonstruc}
	\mathcal{M}^{-1}\mathcal{A}=\begin{pmatrix}
		I&\Psi\\
		{\bf{0}}&\Phi
	\end{pmatrix},
\end{align}
where, $\Psi \in \mathbb{R}^{p\times(n+q)}$ and $\Phi \in \mathbb{R}^{(n+q)\times(n+q)}$. This structure implies that the eigenvalue $1$ has an algebraic multiplicity of at least $p$, and the remaining $n+q$ eigenvalues are located in the spectrum of $\Phi$. Let $\mu_k,$ be an arbitrary eigenvalue of $\Phi$. Then, the characteristic polynomial of the preconditioned matrix $\mathcal{M}^{-1}\mathcal{A}$ is
\begin{align}\label{charpol}
(\mathcal{M}^{-1}\mathcal{A}-I)^p\prod_{k=1}^{n+q}(\mathcal{M}^{-1}\mathcal{A}-\mu_k I).
\end{align}
Now, consider the polynomial $(\mathcal{M}^{-1}\mathcal{A}-I)\displaystyle\prod_{k=1}^{n+q}(\mathcal{M}^{-1}\mathcal{A}-\mu_k I)$ of degree $n+q+1$, which divides the characteristic polynomial \eqref{charpol}. It can be rewritten as
\begin{align*}
(\mathcal{M}^{-1}\mathcal{A}-I)\prod_{k=1}^{n+q}(\mathcal{M}^{-1}\mathcal{A}-\mu_k I)=
\begin{pmatrix}
{\bf{0}}&\Psi \displaystyle\prod_{k=1}^{n+q}(\Phi-\mu_kI)\\
{\bf{0}}&(\Phi-I)\displaystyle\prod_{k=1}^{n+q}(\Phi-\mu_kI)
\end{pmatrix},
\end{align*}
According to the Cayley-Hamilton theorem, this polynomial is equal to the zero matrix. Thus, the degree of the minimal polynomial of $\mathcal{M}^{-1}\mathcal{A}$ is at most $n+q+1.$ This completes the proof.
\end{proof}

\begin{corollary}
	The upper bound on the degree of the minimal polynomial, established in Theorem \ref{Theorminimal}, has a direct implication for the convergence of the GMRES method. Since the degree of the minimal polynomial is equal to the dimension of the Krylov subspace $\mathcal{K}(\mathcal{M}^{-1}\mathcal{A},c)$ for a generic vector $c$, the upper bound directly implies that the GMRES algorithm will terminate in at most $n+q+1$ iterations (in exact arithmetic)\cite[Proposition 6.1]{SaadBook}.
\end{corollary} 
\section{Numerical experiments}\label{sec.4}
In this section, we present some numerical examples to demonstrate the effectiveness and robustness of the proposed preconditioners for solving the ILS problem \eqref{eq:ILS}.
It is well-known that applying a preconditioned Krylov subspace method, such as GMRES, requires solving a linear system of the form $\mathcal{M}z=r$ at each iteration. For large-scale and sparse problems, solving these inner systems exactly may  be computationally impractical. To overcome this, we solve the inner systems inexactly using a few iterations of CG method. Therefore, using the Flexible GMRES (FGMRES) method is an ideal choice \cite{FGMRES,SaadBook}. So, we employ the preconditioned version of it in our experiments.

We compare our proposed IBS preconditioners with two recent effective preconditioners, denoted by BS2 and BUT. The choice of BS2 is motivated by the observation in \cite{Xin} that the BS1, BS2, and BS3 schemes exhibit similar performance, with BS2 being slightly more efficient. For all experiments, the initial guess is the zero vector, and the approximation matrix $\hat{P}$ is chosen as 
$\hat{P} = \alpha I+P,$
with 
$$\alpha={\|A_1\|_1^2}.$$ 
All computations were performed in \textsc{Matlab} R2022a on a laptop with a 2.30 GHz Intel(R) Core(TM) i3-2348 CPU and 4 GB of RAM, running a Windows 10 operating system. 

The iteration of FGMRES method is terminated as soon as the relative residual norm
$$
\text{RES} = \frac{\|\tilde{b}-\mathcal{A}\mathbf{x}_k\|_2}{\|\tilde{b}\|_2},
$$
where $\textbf{x}_{k}$ is the approximate solution at iteration $k$, satisfies $\text{RES}<10^{-8}$, or when the number of iterations exceeds 2000.
Furthermore, the inner CG iteration is terminated once the relative residual norm falls below the tolerance of $10^{-3}$, or after a maximum of 1000 iterations.
The performance of all the preconditioned methods is evaluated based on the number of iterations (IT), the algorithm's running time in seconds (CPU), and the relative error of the approximate solution (ERR), defined as
$$
\text{ERR} = \frac{\| x_k - x^*\|_2}{\|x^*\|_2},
$$
where $x_k$ is the approximate solution, extracted from $\mathbf{x}_k$, and $x^*$ is exact solution for the ILS problem \eqref{eq:ILS}. The reported IT and CPU values are the average over five consecutive runs. In all tables, the notation $\dagger$ indicates that the corresponding method failed to converge to the solution of system \eqref{eq:block3x3} in 2000 iterations.

\begin{example}\label{ex:tols_model} \rm
In our first set of experiments, we evaluate the performance of the proposed IBS preconditioners on a series of matrices obtained from the well-known Matrix Market website \url{https://math.nist.gov/MatrixMarket/}. Specifically, we use four matrices TOLS340, TOLS1090, TOLS2000, and TOLS4000. These matrices arise from the stability analysis of an aircraft flight model, a problem studied at CERFACS. The key properties of these test matrices are summarized in Table \ref{tab:tols_properties}. It should be noted that all matrices were normalized to have a unit 1-norm prior to use. This scaling procedure is a standard practice to ensure a fair comparison between problems of different scales. In this case the value of $\alpha$ would equal 1. For each test problem, we set the matrix $A_1$ to be one of the TOLS matrices, thus setting $p=n$. We then construct the matrix $A_2$ by setting $q=10000$ and $A_2 = 6 I_{q \times n}$, where $I_{q \times n}$ is a $q \times n$ rectangular matrix with ones on its main diagonal and zeros elsewhere. The right-hand side vectors $b_1$ and $b_2$ are chosen as vectors of all ones.
	
\begin{table}
		\centering
		\caption{Properties of the test matrices for Example \ref{ex:tols_model}.} \label{tab:tols_properties}
		\vspace{0.25cm}
		\begin{tabular}{|c|ccc| } \hline
	matrix           &  ~~~~$n$~~~~    &  ~~~~ $nnz$ ~~~~  & ~~~~ Cond ~~~~ \\ \hline
	TOLS 340    &   340       &   2196        &        $2.4e+05$              \\ 
	TOLS1090   &   1090     &    3546       &        $2.1e+06$              \\ 
	TOLS2000   &   2000     &   5184        &        $6.9e+06$             \\ 
	TOLS4000   &   4000     &   8784        &        $2.7e+07$               \\ \hline		
\end{tabular}
	\end{table}
	
The numerical results are reported in Table \ref{tab:ex1_results}. The results clearly demonstrate the effectiveness of all the IBS preconditioners. A closer look at the table reveals two key observations. First, all four IBS preconditioners significantly outperform the competing BUT and BS2 preconditioners in every test case. Second, a comparison among the IBS preconditioners shows that the IBS2 and IBS4 variants consistently achieve the best performance in terms of both iteration count and CPU time.

\begin{table}[!hp]
	\centering
	\caption{Numerical results for Example \ref{ex:tols_model}.}
	\label{tab:ex1_results}
\vspace{0.25cm}
\begin{tabular}{|c|c|c|c|c|c|c|c|}
	\hline
		$m \times n$ &   & IBS1 & IBS2 & IBS3 & IBS4 & BS2 & BUT \\
		\hline
		\multirow{4}{*}{$10340 \times 340$} 
		& IT       & 40       & 31       & 40       & 31       & 324       & 361 \\
		& CPU      & 0.29   & 0.23   & 0.29   & 0.23   & 11.93   & 14.92 \\
		& RES      & 9.89e-10 & 5.69e-10 & 9.85e-10 & 2.94e-10 & 9.97e-10  & 9.37e-10 \\
		& ERR      & 3.32e-09 & 2.77e-09 & 3.35e-09 & 1.41e-09 & 4.81e-09  & 4.89e-09 \\
       \hline
		\multirow{4}{*}{$11090 \times 1090$} 
		& IT       & 41       & 31       & 41       & 31       & 1570      & 1561 \\
		& CPU      & 0.33   & 0.25   & 0.33   & 0.25   & 266.01  & 264.87 \\
		& RES      & 3.63e-10 & 8.19e-10 & 3.60e-10 & 3.69e-10 & 9.92e-10  & 9.54e-10 \\
		& ERR      & 6.68e-10 & 2.33e-09 & 6.71e-10 & 1.03e-09 & 8.20e-08  & 9.24e-08 \\
        \hline
		\multirow{4}{*}{$12000 \times 2000$} 
		& IT       & 41       & 31       & 41       & 31       &      &   \\
		& CPU      & 0.35   & 0.27   & 0.37   & 0.26   & $\dagger$        & $\dagger$ \\
		& RES      & 3.70e-10 & 8.90e-10 & 3.68e-10 & 3.79e-10 &          &   \\
		& ERR      & 5.42e-10 & 2.00e-09 & 5.45e-10 & 8.33e-10 &          &   \\
		\hline
		\multirow{4}{*}{$14000 \times 4000$} 
		& IT       & 40       & 31       & 40       & 31       & 131       & 1975 \\
		& CPU      & 0.38   & 0.29   & 0.38   & 0.29   & 12.82   & 687.31 \\
		& RES      & 9.92e-10 & 9.08e-10 & 9.87e-10 & 3.68e-10 & 1.45e-10  & 9.51e-10 \\
		& ERR      & 1.22e-09 & 1.63e-09 & 1.23e-09 & 6.38e-10 & 6.02e-11  & 1.16e-08 \\ 
		\hline
	\end{tabular}
\end{table}
\end{example}
\bigskip
\begin{example}
	\label{ex:sherman_model} \rm
	In our second set of experiments, we test the proposed preconditioners on problems arising from oil reservoir simulation. We use five matrices SHERMAN1, SHERMAN2, SHERMAN3, SHERMAN4, and SHERMAN5, which are derived from Matrix Market. These matrices are known to be challenging due to their structure and condition number. The key properties of these test matrices are summarized in Table \ref{tab:sherman_properties}. 
	
\begin{table}
		\centering
		\caption{Properties of the test matrices for Example \ref{ex:sherman_model}.} \label{tab:sherman_properties}
		\vspace{0.25cm}
	\begin{tabular}{|c|ccc| } \hline
	matrix           &  ~~~~$n$~~~~    &  ~~~~ $nnz$ ~~~~  & ~~~~ Cond ~~~~ \\ \hline
	SHERMAN1    &   1000       &   3750        &        $2.3e+04$              \\ 
	SHERMAN2   &   1080     &    23094       &        $1.4e+12$              \\ 
	SHERMAN3   &   5005     &   20033        &        $6.9e+16$             \\ 
	SHERMAN4   &   1104     &   3786        &        $7.2e+03$               \\ 
	SHERMAN5   &   3312     &   20793      &        $3.9e+05$  \\ \hline		
\end{tabular}
\end{table}

For each test problem, we set the matrix $A_1$ to be one of the Sherman matrices. Thus we have $p=n$. As in the previous example, all matrices were normalized to have a unit 1-norm before using them. For this set of experiments, we increase the size of the second block by setting $q=15000$ and construct the matrix $A_2 = 6 I_{q \times n}$. The right-hand side vectors $b_1$ and $b_2$ are again chosen as vectors of all ones. 
	
The numerical results are reported in Table \ref{tab:ex2_results}. A key observation from the results is that all proposed IBS preconditioners converge rapidly, requiring significantly fewer iterations and less CPU time than the competing methods. Furthermore, the small relative errors, ERR, indicate that this convergence to the solution of the system \eqref{eq:block3x3} also leads to an accurate approximation for the exact solution of the ILS problem. In contrast, while the BS2 and BUT methods sometimes converge (as seen by their small RES values), their corresponding ERR values are close to 1, indicating a complete failure to approximate the exact solution of ILS problem. This illustrates a critical advantage of our proposed IBS preconditioners.
\begin{table}[!ht]
	\centering
	\caption{Numerical results for Example \ref{ex:sherman_model}.}
	\label{tab:ex2_results}
	\vspace{0.25cm}
	\begin{tabular}{|c|c|c|c|c|c|c|c|}
		\hline
		$   m \times n $ &  & IBS1 & IBS2 & IBS3 & IBS4 & BS2 & BUT  \\
		\hline
		\multirow{4}{*}{$ 16000 \times 1000 $} 
		& IT       & 28       & 22       & 28       & 20       & 313       & 323  \\
		& CPU      & 0.26   & 0.24   & 0.27   & 0.22   & 24.12   & 24.79  \\
		& RES      & 5.05e-10 & 8.18e-10 & 5.04e-10 & 4.95e-10 & 9.07e-10  & 8.77e-10  \\
		& ERR      & 1.19e-09 & 7.53e-10 & 1.19e-09 & 1.45e-09 & 0.76    & 0.78  \\
		\hline
		\multirow{4}{*}{$ 16080 \times 1080 $} 
		& IT       & 49       & 35       & 49       & 37       & 907       & 907  \\
		& CPU      & 0.58  & 0.37   & 0.58   & 0.38   & 186.98  & 187.25  \\
		& RES      & 4.37e-10 & 9.12e-10 & 4.25e-10 & 8.48e-10 & 3.66e-11  & 3.66e-11  \\
		& ERR      & 1.46e-09 & 3.75e-09 & 1.47e-09 & 2.76e-09 & 1.00    & 1.00  \\
		\hline
		\multirow{4}{*}{$ 20005 \times 5005 $} 
		& IT       & 26       & 21       & 26       & 19       & 127       & 77  \\
		& CPU      & 0.32   & 0.26   & 0.32   & 0.24   & 25.37   & 15.10  \\
		& RES      & 8.14e-10 & 3.13e-10 & 8.14e-10 & 4.50e-10 & 2.96e-10  & 7.05e-10  \\
		& ERR      & 1.03e-09 & 6.88e-10 & 1.03e-09 & 7.00e-10 & 0.39   & 0.72  \\
		\hline
		\multirow{4}{*}{$ 16104 \times 1104 $} 
		& IT       & 24       & 19       & 24       & 19       & 267       & 371  \\
		& CPU      & 0.26   & 0.22   & 0.26   & 0.22   & 13.36   & 23.47  \\
		& RES      & 4.28e-10 & 5.69e-10 & 4.27e-10 & 2.21e-10 & 8.84e-10  & 7.43e-10  \\
		& ERR      & 1.16e-09 & 2.35e-09 & 1.16e-09 & 9.12e-10 & 1.91e-08  & 4.12e-08  \\
		\hline
		\multirow{4}{*}{$ 18312 \times 3312 $} 
		& IT       & 32       & 26       & 32       & 27       & 1673      & 723  \\
		& CPU      & 0.35   & 0.30   & 0.35   & 0.30   & 623.44  & 161.77  \\
		& RES      & 9.28e-10 & 5.63e-10 & 9.21e-10 & 2.36e-10 & 9.96e-10  & 8.79e-10  \\
		& ERR      & 1.82e-09 & 1.56e-09 & 1.83e-09 & 6.42e-10 & 2.80    & 1.31  \\
		\hline
	\end{tabular}
\end{table}
\end{example}

\begin{example}
\label{ex:hilbert} \rm
	In our final example, we test the robustness of the proposed preconditioners on problems constructed from the severely ill-conditioned Hilbert matrix, $H_n$. We construct a set of test problems by setting $A_1 = H_n$ and $p=q=n$ for various increasing sizes, $n = 400, 800, 1200, 1600, 10000$. The matrix $A_2$ is set to $A_2 = 0.7I_n$, while all other experimental settings are kept the same as in the previous examples.
		
		The numerical results are presented in Table \ref{tab:ex3_results}. It clearly shows the superiority of our proposed preconditioners. Due to the extreme ill-conditioning, both the BUT and BS2 preconditioned methods fail to converge in all the large test cases. In contrast, all the proposed IBS preconditioners solve these challenging problems efficiently. This experiment demonstrates the superior robustness of our inexact preconditioning technique when dealing with severely ill-conditioned systems. Consistent with our previous findings, the IBS2 and IBS4 schemes again perform the best. It can be observed that IBS4 is slightly faster in terms of CPU time.

\begin{table}[!ht]
	\centering
	\caption{Numerical results for Example \ref{ex:hilbert}.}
	\label{tab:ex3_results}
\vspace{0.25cm}
\begin{tabular}{|c|c|c|c|c|c|c|c|}
	\hline
		$m \times n$ &  & IBS1 & IBS2 & IBS3 & IBS4 & BS2 & BUT \\
		\hline
		\multirow{4}{*}{$800 \times 400$} 
		& IT       & 13       & 10       & 13       & 10       & 80        & 96 \\
		& CPU      & 0.15   & 0.14   & 0.15   & 0.14   & 7.01    & 9.78 \\
		& RES      & 1.97e-10 & 1.00e-13 & 2.17e-10 & 2.01e-14 & 1.54e-11  & 7.24e-10 \\
		& ERR      &2.72e-10&1.26e-13&3.12e-10&6.25e-14&5.80e-10&1.46e-08\\
		\hline
		\multirow{4}{*}{$1600 \times 800$} 
		& IT       & 14       & 10       & 14       & 10       & 98        & 85 \\
		& CPU      & 0.26   & 0.23   & 0.28   & 0.23   & 91.85   & 79.24 \\
		& RES      & 1.85e-12 & 3.32e-12 & 1.85e-12 & 1.79e-12 & 2.30e-10  & 6.92e-10 \\
		& ERR      &1.52e-11&1.71e-11&1.52e-11&1.69e-11&1.82e-09&7.32e-09\\
		\hline
		\multirow{4}{*}{$2400 \times 1200$} 
		& IT       & 14       & 10       & 14       & 10       & 100       & 82 \\
		& CPU      & 0.41   & 0.35   & 0.44   & 0.34   & 217.55  & 178.02\\
		& RES      & 2.44e-11 & 2.74e-11 & 2.45e-11 & 2.00e-10 & 8.56e-11  & 5.27e-10 \\
		& ERR      &2.01e-10&2.27e-10&2.02e-10&1.56e-10&4.71e-10&6.18e-09\\
		\hline
		\multirow{4}{*}{$3200 \times 1600$} 
		& IT       & 14       & 10       & 14       & 10       & 92        & 96 \\
		& CPU      & 0.66   & 0.51   & 0.67   & 0.50   & 369.12  & 377.64 \\
		& RES      & 1.41e-10 & 1.59e-10 & 1.39e-10 & 1.31e-10 & 9.93e-10  & 2.07e-10 \\
		& ERR      &1.16e-09&1.29e-09&1.15e-09&9.21e-10&3.21e-08&4.98e-09\\
		\hline
		\multirow{4}{*}{$20000 \times 10000$} 
		& IT       & 16       & 11       & 15       & 11       &     &  \\
		& CPU      & 24.90  & 18.22  & 24.61  & 18.18  & $\dag$    &   $\dag$  \\
		& RES      & 1.92e-10 & 1.22e-10 & 8.85e-10 & 8.98e-11 &   &  \\
		& ERR      &1.59e-09&1.01e-09&1.62e-09&7.35e-10&&\\
		\hline
	\end{tabular}
\end{table}
\end{example}

\section{Conclusions} \label{sec.5}
In this paper, we have proposed several inexact block-splitting preconditioners, denoted by IBS, for a class of large-scale indefinite least squares problems. A detailed theoretical analysis was provided, including the convergence conditions for the corresponding stationary methods and a detailed study of the eigenstructure of the preconditioned matrices. This analysis led to a theoretical bound on the number of GMRES iterations when our preconditioners are applied. Finally, numerical experiments demonstrated the efficiency and robustness of our proposed preconditioners. They performed well particularly  on large sparse, and ill-conditioned problems, where many other preconditioners often fail and always perform worse. This shows the great potential as a practical tool for real-world ILS problems.

\section*{Acknowledgements}
 The authors sincerely thank the reviewers for their meticulous review of the paper and for their constructive comments and suggestions.

\section*{Declarations}
\textit{Data availability:} Not applicable.\\
\textit{Funding:} Not applicable.\\
\textit{Conflict of interest:} The authors declare that they have no conflict of interest.\\
\textit{Author contributions:} Both theoretical and numerical experiments have been conducted by the authors.

\bibliographystyle{plain}

\end{document}